\documentclass[11pt,a4paper]{amsart}
\usepackage[utf8]{inputenc}
\usepackage[english]{babel}
\usepackage{pst-grad} 
\usepackage{pst-plot} 
\usepackage{pstricks}
\usepackage{amsmath}
\usepackage{amssymb}
\usepackage{mathrsfs}
\usepackage{amsthm}
\usepackage[dvips]{graphicx}
\usepackage{epsfig}

\newcommand{\RR}{\mathbb R}

\newcommand{\ZZ}{\mathbb Z}

\newcommand{\pat}{\partial_t}
\newcommand{\pax}{\partial_x}

\newcommand{\vertiii}[1]{{\left\vert\kern-0.25ex\left\vert\kern-0.25ex\left\vert #1 
    \right\vert\kern-0.25ex\right\vert\kern-0.25ex\right\vert}}

\usepackage[pdfauthor={...},pdftitle={...},bookmarks,colorlinks]{hyperref}

\newcounter{comentcount}
\newcounter{teocount}
\newtheorem{lem}{Lemma}
\newtheorem{prop}{Proposition}

\newtheorem{thm}[teocount]{Theorem}

\newtheorem{rem}{Remark}

\title[]{On a drift-diffusion system for semiconductor devices}

\author[R. Granero-Belinch\'{o}n]{Rafael Granero-Belinch\'{o}n}
\email{rgranero@math.ucdavis.edu}
\address{Department of Mathematics, University of California, Davis, CA 95616, USA}

\begin{document}

\begin{abstract}
In this note we study a fractional Poisson-Nernst-Planck equation modeling a semiconductor device. We prove several decay estimates for the Lebesgue and Sobolev norms in one, two and three dimensions. We also provide the first term of the asymptotic expansion as $t\rightarrow\infty$.
\end{abstract}

\maketitle 

\section{Introduction}
We consider the drift-diffusion system given below: 
	\begin{equation}\label{eq1}\begin{cases}
		\pat u + (-\Delta)^{\frac{\alpha}{2}} u + \nabla \cdot (u \nabla \psi) = 0, \text{ for }(x,t) \in \RR^d\times\RR^+ \\
		\pat v + (-\Delta)^{\frac{\beta}{2}} v - \nabla \cdot (v \nabla \psi) = 0, \text{ for }(x,t) \in \RR^d\times\RR^+ \\
		\Delta \psi = u-v , \text{ for }(x,t) \in \RR^d\times\RR^+ \\
		u(x,0)=u_0, v(x,0)=v_0, x \in \RR^d
	\end{cases}\end{equation}
where $u$, $v$, and $\psi$ are functions of space and time, the dimension $d\in \ZZ^+$ with $d\leq3$, $0<\alpha,\beta<2$, and, if we denote the Fourier transform of the function $\phi$ by $\hat{\phi}$, then the fractional Laplacian is defined by
$$
\widehat{(-\Delta)^{\frac{\alpha}{2}} \phi} = |\xi|^\alpha\hat{\phi}.
$$
The unknown functions $u(x,t)$ and $v(x,t)$ represent the density of electrons and positive holes in the semiconductor, respectively. Finally, the function $\psi$ models the electromagnetic potential due to charges in a semiconductor. The fractional Laplacians are related to random trajectories, generalizing the concept of Brownian motion, which may contain jump discontinuities (the, so-called, $\alpha$-stable L\'evy processes). As an electron in a semiconductor may jump from a dopant into another, a nonlocal diffusion akin to the fractional Laplacian arises naturally.

\subsection{Prior results on (\ref{eq1})}
Mock \cite{mock1975asymptotic} first considered the drift-diffusion system \eqref{eq1} with $\alpha=\beta=2$ on a bounded domain with the Neumann boundary condition (see also He, Gamba, Lee \& Ren,\cite{he2013modeling} and Liu \& Wang \cite{liu2010poisson}) A similar equation has been studied by Rodr\'iguez \& Ryzhik in a very different context \cite{rodriguez2013exploring}. Fang \& Ito \cite{fang1995global} proved the existence of a global weak solution in this bounded domain (see also the work by Bothe, Fischer, Pierre, \& Rolland,\cite{bothe2014global} and  Hineman, \& Ryham \cite{hineman2015very}). The asymptotic behaviour of the solution in the case $\alpha=\beta=2$ was studied by Jungel \cite{jungel1995qualitative} and Biler \& Dolbeault \cite{biler2000long}. Kurokiba \& Ogawa \cite{kurokiba2008well} and Kurokiba, Nagai \& Ogawa \cite{kurokiba2006uniform} proved the existence and uniqueness of strong solutions to the Cauchy problem. Kawashima \& Kobayashi \cite{kobayashi2008decay} derived the optimal decay estimate by applying a weighted energy method and found an asymptotic result as $t \rightarrow \infty$. In presence of an incompressible, viscous fluid, system \eqref{eq1} was studied by Schmuck \cite{schmuck2009analysis}, by Zhao, Deng \& Cui \cite{zhao2010global, zhao2011well, deng2011well}, by Bothe, Fischer, \& Saal \cite{bothe2014global2}. Very recently, Kinderlehrer, Monsaingeon, \& Xu provided a new approach to system \eqref{eq1} using that system \eqref{eq1} is a gradient flow driven by a $L\log L-$type free energy \cite{kinderlehrer2015wasserstein}. Each of these studies restricted their conclusions to $\alpha=\beta=2$. The case of nonlinear diffusion has been considered by Zinsl \cite{zinsl2015exponential} 

When $v_0\equiv0$ (so the equation for $v$ is dropped), the fractional case $0<\alpha\leq2$ of \eqref{eq1} has been studied by several authors. Yamamoto \cite{yamamoto2010asymptotic} obtained the asymptotic behavior in the local case $\alpha=2$. Yamamoto \cite{yamamoto2012large} proceeded similarly, but derived the asymptotic expansion of the solution with the fractional Laplacian in the subcritical regime $1<\alpha<2$. Yamamoto, Kato \& Sugiyama, \cite{yamamoto2014existence} showed the well-posedness and real analytic of the critical case corresponding to $\alpha=1$. Sugiyama, Yamamoto \& Kato \cite{sugiyama2015local} studied local and global existence and uniqueness of the system with the fractional Laplacian, focusing primarily on the supercritical and critical cases $0<\alpha<1$ and $\alpha=1$, respectively. Yamamoto \& Sugiyama \cite{yamamotoasymptotic, yamamoto2015asymptotic} then derived lower bounds on the decay rates of a solution to the drift-diffusion system with the fractional Laplacian $0<\alpha\leq1$ and obtained the asymptotic behavior of the solution as $ t \rightarrow \infty$. Similar systems arising in different contexts have been studied also by Li, Rodrigo \& Zhang \cite{li2010exploding}, Escudero \cite{escudero2006fractional}, Bournaveas \& Calvez \cite{bournaveas2010one}, Biler \& Karch \cite{biler2010blowup}, Biler \& Wu \cite{BilerWu}, Biler, Karch \& Woyczy{\'n}ski \cite{biler2001critical} Biler \& Woyczy{\'n}ski \cite{biler1998global, biler2007general}, Zhao \cite{zhao2015well, zhao2015optimal}, Ascasibar, Granero-Belinch\'on \& Moreno \cite{AGM} and Burczak \& Granero-Belinch\'on \cite{BG3, BG4}.

The fractional case $1<\alpha=\beta<2$ of \eqref{eq1} with general $v_0$ has been studied by Ogawa \& Yamamoto \cite{ogawa2009asymptotic}. In particular, these authors proved the global existence and the asymptotic behavior of solutions. To the best of our knowledge, this is the only result concerning \eqref{eq1}. Thus, by studying \eqref{eq1}, this paper generalizes the current results in \cite{ogawa2009asymptotic} in two different aspects:

\begin{enumerate}
\item it allows for diffusions with different strengths for $u$ and $v$ \emph{i.e.} $\alpha$ is not necessarily equal to $\beta$. The cases $\alpha\neq\beta$ and $\alpha=\beta$ present several differences at the level of the $H^2$ Sobolev norm and some \emph{closeness} hypothesis needs to be imposed (see Theorem \ref{thdecay3}). 
\item it allows for diffusions in the whole range $0<\alpha,\beta<2$. In particular, our work covers the supercritical and critical range $0<\alpha,\beta\leq1$.
\end{enumerate}

\subsection{Preliminaries}
\subsubsection{Singular integral operators}
We write $\Lambda^\alpha =(-\Delta)^{\frac{\alpha}{2}}$, \emph{i.e.}
\begin{equation}\label{eq:7}
\widehat{\Lambda^{\alpha}u}(\xi)=|\xi|^\alpha \hat{u}(\xi)
\end{equation}
where $\hat{\cdot}$ denotes the usual Fourier transform. As a singular integral operator, the operator $\Lambda^\alpha$ possesses the kernel\begin{align}\label{eq:8}
\Lambda^\alpha u(x)&=c_{\alpha,d}\text{P.V.}\int_{\RR^d}\frac{u(x)-u(x-\eta)d\eta}{|\eta|^{d+\alpha}},
\end{align}
with
$$
c_{\alpha,d}=\frac{4^s\Gamma(d/2+\alpha)}{\pi^{d/2}|\Gamma(-\alpha)|}>0,
$$
where 
$$
\Gamma(z)=\int_0^\infty t^{z-1}e^{-t}dt
$$ 
is the $\Gamma$ function.

\subsubsection{Functional spaces}
We write $L^p(\RR^d)$ for the usual Lebesgue spaces
$$
L^p(\RR^d)=\left\{u\text{ measurable s.t. }\int_{\RR^d}|u(x)|^pdx<\infty\right\},
$$
with norm
$$
\|u\|_{L^p}^p=\int_{\RR^d}|u(x)|^pdx.
$$

We write $H^s(\RR^d)$ for the usual $L^2$-based Sobolev spaces:
$$
H^s(\RR^d)=\left\{u\in L^2(\RR^d) \text{ s.t. }(1+|\xi|^s)\hat{u}\in L^2(\RR^d)\right\},
$$
with the norm
$$
\|u\|_{H^s}^2=\|u\|_{L^2}^2+\|u\|_{\dot{H}^s}^2, \quad \|u\|_{\dot{H}^s}=\|\Lambda^s u\|_{L^2}.
$$

\subsection{Plan of the paper}
This note is organized as follows: in section \ref{sec1} we state our results. In section \ref{sec2} we prove the decay in the $L^p$ spaces. In sections \ref{sec3}, \ref{sec4} and \ref{sec5}, we prove the decay of the Sobolev norms. Then, in section \ref{sec6}, we provide the first term in the asymptotic expansion. Finally, in Appendix \ref{sec7}-\ref{sec9} we provide certain inequalities and estimates for fractional Laplacian that are used in the paper and may be interesting by themselves.

\section{Main results}\label{sec1}
Our first result concerns the global existence and decay of the solutions to \eqref{eq1}:\begin{thm}\label{thdecay}Let $0<\alpha,\beta<2$, $d\in \ZZ^+$ with $d\leq3$, be fixed constants and
$$
u_0(x),v_0(x)\in L^1(\RR^d)\cap H^{4}(\RR^d)
$$ 
be the initial data. Then there exists $(u(x,t),v(x,t))$ a global smooth solution to \eqref{eq1} satisfying
$$
u\in L^\infty(0,T;L^1(\RR^d)\cap H^4(\RR^d))\cap L^2(0,T;H^{4+\alpha/2}(\RR^d)),
$$
$$
v\in L^\infty(0,T;L^1(\RR^d)\cap H^4(\RR^d))\cap L^2(0,T;H^{4+\beta/2}(\RR^d)),
$$
for every $0<T<\infty$. Furthermore, the functionals
$$
\mathcal{F}_p[u(t),v(t)]:=\|u(t)\|^p_{L^p}+\|v(t)\|^p_{L^p},\;1\leq p<\infty,
$$
$$
\mathcal{F}_\infty[u(t),v(t)]:=\|u(t)\|_{L^\infty}+\|v(t)\|_{L^\infty},
$$
verify
$$
\mathcal{F}_p[u(t),v(t)]\leq \mathcal{F}_p[u_0,v_0],\:1\leq p\leq\infty,
$$
and there exist constants $K$ and $C_\infty$ such that 
$$
\mathcal{F}_\infty[u,v]\leq \frac{\mathcal{F}_\infty[u_0,v_0]}{\left(1+Kt\right)^{d/\max\{\alpha,\beta\}}},
$$
$$
\mathcal{F}_p[u,v]\leq \left(\|u_0\|_{L^1}+\|v_0\|_{L^1}\right)\frac{C_\infty^{p-1}}{\left(1+t\right)^{\frac{d}{\max\{\alpha,\beta\}}(p-1)}},
$$
$$
\|u(t)\|_{L^p}\leq \frac{\|u_0\|_{L^1}^{\frac{1}{p}}C_\infty^{1-\frac{1}{p}}}{\left(1+t\right)^{\frac{d}{\max\{\alpha,\beta\}}(1-\frac{1}{p})}},
$$
$$
\|v(t)\|_{L^p}\leq \frac{\|v_0\|_{L^1}^{\frac{1}{p}}C_\infty^{1-\frac{1}{p}}}{\left(1+t\right)^{\frac{d}{\max\{\alpha,\beta\}}(1-\frac{1}{p})}}.
$$
\end{thm}

\begin{rem}In the case where the smooth initial data is 
$$
u_0(x),v_0(x)\in L^p(\RR^d),\; 1<p<\infty, 
$$ 
following the proof of Theorem \ref{thdecay}, we have the pointwise estimates
\begin{align*}
\Lambda^\alpha u(x_t) &\geq c(d,\alpha,p)\frac{u(x_t)^{1+\alpha p/d}}{\|u(t)\|^{\alpha p/d}_{L^p}}\geq c(d,\alpha,p)\frac{u(x_t)^{1+\alpha p/d}}{\mathcal{F}_p(u_0,v_0)^{\alpha /d}}\\
\Lambda^\beta v(y_t) &\geq c(d,\beta,p)\frac{v(y_t)^{1+\beta p/d}}{\|v(t)\|^{\beta p/d}_{L^p}}\geq c(d,\beta,p)\frac{v(y_t)^{1+\beta p/d}}{\mathcal{F}_p(u_0,v_0)^{\beta /d}},
\end{align*}
where $x_t$ and $y_t$ are such that
$$
\|u(t)\|_{L^\infty}=u(x_t,t),\;\|v(t)\|_{L^\infty}=v(y_t,t).
$$
Thus, instead of \eqref{Decayinfty1}, we have that
\begin{equation}\label{Decayinftyp}
\frac{d}{dt} \mathcal{F}_\infty[u,v] \leq -c(d,\alpha,p)\frac{u(x_t)^{1+\alpha p/d}}{\mathcal{F}_p(u_0,v_0)^{\alpha /d}}-c(d,\beta,p)\frac{v(y_t)^{1+\beta p/d}}{\mathcal{F}_p(u_0,v_0)^{\beta /d}}.
\end{equation}
\end{rem}

Our second result studies the behavior of Sobolev spaces $H^s$ for $0<s\leq1$. 

\begin{thm}\label{thdecay2}Let $0<\alpha,\beta<2$, $d\in \ZZ^+$ with $d\leq3$, be fixed constants and
$$
u_0(x),v_0(x)\in L^1\cap H^4
$$ 
be the initial data. Then, there exists a constant $C$ such that the solution $(u(x,t),v(x,t))$ to \eqref{eq1} verifies
$$
\|u(t)\|_{\dot{H}^{s}}+\|v(t)\|_{\dot{H}^{s}}\leq \frac{C}{(1+t)^{\frac{d}{\max\{\alpha,\beta\}}\frac{1-s}{2}}},\,\forall\,t\geq 0,\,0\leq s\leq1.
$$
\end{thm}

Our third result regards the higher Sobolev norm $H^s$, $1\leq s\leq 2$ and imposes restrictions on $\alpha$ and $\beta$:
\begin{thm}\label{thdecay3}Let $0<\alpha,\beta<2$, $d\in\ZZ^+$, $d\leq3$, be fixed constants such that
\begin{equation}\label{condition}
\frac{2d}{4+3\min\{\alpha,\beta\}}< 1,
\end{equation}
\begin{equation}\label{condition2}
\frac{\min\{\alpha,\beta\}}{\max\{\alpha,\beta\}}\frac{d}{4}\left(2+\frac{2d}{4+3\max\{\alpha,\beta\}-2d}\right)>1,
\end{equation}
\begin{equation}\label{condition3}
\Gamma=\frac{d}{4+3\min\{\alpha,\beta\}-d}\left(2+\frac{4|\alpha-\beta|}{4+3\min\{\alpha,\beta\}}\right)\leq2,
\end{equation}
and
$$
u_0(x),v_0(x)\in L^1\cap H^4
$$ 
be the initial data. Furthermore, when $\Gamma<2$, we assume that
\begin{equation}\label{condition4}
\frac{4+3\max\{\alpha,\beta\}}{4+3\min\{\alpha,\beta\}}<1+\frac{d}{4+3\max\{\alpha,\beta\}-d},
\end{equation}
and
\begin{equation}\label{condition5}
\frac{\min\{\alpha,\beta\}}{\max\{\alpha,\beta\}}\frac{d}{2}\geq1.
\end{equation}
Then, there exists a constant $C$ such that the solution $(u(x,t),v(x,t))$ to \eqref{eq1} verifies
$$
\|u(t)\|_{\dot{H}^{s}}+\|v(t)\|_{\dot{H}^{s}}\leq \frac{C}{(1+t)^{\frac{d}{\max\{\alpha,\beta\}}\frac{2-s}{4}}},\,\forall\,t\geq 0,\,0\leq s\leq 2.
$$
\end{thm}
Notice that this result imposes restrictions on the difference $\alpha-\beta$. This result suggests that a big disparity in the strengths of the diffusive operators may lead to obstructions in higher Sobolev norms. 

Our next theorem concerns the case of arbitrarily large Sobolev norms:
\begin{thm}\label{thdecay4}Let $0<\alpha,\beta<2$, $d\in \ZZ^+$ with $d\leq3$, be fixed constants and
$$
u_0(x),v_0(x)\in L^1\cap H^s,\;s\geq2,\,s\in\RR
$$ 
be the initial data. Assume that $\alpha,\beta$ and $d$ satisfy the same hypothesis as in Theorem \ref{thdecay3}. Then, there exists a constant $C$ such that the solution $(u(x,t),v(x,t))$ to \eqref{eq1} verifies
$$
\|u(t)\|_{\dot{H}^{r}}+\|v(t)\|_{\dot{H}^{r}}\leq \frac{C}{(1+t)^{\frac{d}{\max\{\alpha,\beta\}}\frac{s-r}{2s}}},\,\forall\,t\geq 0,\,0\leq r\leq s.
$$
\end{thm}

Finally, we provide the first order asymptotic estimate
\begin{prop}\label{thasym}Let $0<\alpha,\beta<2$, $d\in \ZZ^+$ with $d\leq3$, be fixed constants and
$$
u_0(x),v_0(x)\in L^1\cap H^s,\;s\geq2,\,s\in\RR
$$ 
be the initial data. Then, there exists a constant $C$ such that the solution $(u(x,t),v(x,t))$ to \eqref{eq1} verifies
\begin{align*}
\|u(t)-e^{-t\Lambda^{\alpha}}u_0\|_{L^2}&\leq \frac{C}{(1+t)^{\frac{d-1}{\max\{\alpha,\beta\}}-1}}\\
\|v(t)-e^{-t\Lambda^{\beta}}v_0\|_{L^2}&\leq \frac{C}{(1+t)^{\frac{d-1}{\max\{\alpha,\beta\}}-1}}
\end{align*}
\end{prop}

\section{Proof of Theorem \ref{thdecay}: Global existence and $L^p$ decay estimates}\label{sec2}
\textbf{Step 1: Local existence} The local existence and uniqueness follows from standard methods (see for instance \cite{AGM}).

\textbf{Step 2: Boundedness in $L^p$} First notice that, given $u_0(x) \geq 0$ and $v_0(x) \geq 0$, we have that $ u(t) \geq 0$ and $v(x,t) \geq 0 \ \forall t\geq0$ (this can be shown with a contradiction argument and the use of pointwise methods \cite{BG}). Thus, we have
$$
\frac{d}{dt} \mathcal{F}_1[u,v]=0.
$$
Furthermore, we have the stronger equalities
$$
\|u(t)\|_{L^1}=\|u_0\|_{L^1},\;\|v(t)\|_{L^1}=\|v_0\|_{L^1}.
$$

Consider now the case $1<p<\infty$. Then
\begin{align*}
\frac{d}{dt} \mathcal{F}_p[u,v]&=\frac{d}{dt} \|u(t)\|^p_{L^p}+\frac{d}{dt}\|v(t)\|^p_{L^p} \\
&= \int_{\RR^d} p u(y,t)^{p-1} \pat u(y,t) dy + \int_{\RR^d} p v(y,t)^{p-1} \pat v(y,t) dy \\
&=\int_{\RR^d} p u(y,t)^{p-1} [-\Lambda^\alpha u(y,t)-\nabla \cdot (u(y,t) \nabla \psi)] dy \\ 
		& \quad + \int_{\RR^d} p v(y,t) ^{p-1} [-\Lambda^\beta v(y,t)+ \nabla \cdot (v(y,t)\nabla \psi)] dy \\
\end{align*}
The transport terms are
$$
T_1=-\int_{\RR^d} p u^{p-1} \nabla \cdot (u \nabla \psi) =\int_{\RR^d} p(p-1) u^{p-1} \nabla u \cdot \nabla \psi=-\int_{\RR^d} (p-1) u^{p} \Delta \psi, 
$$
$$
T_2=\int_{\RR^d} p v^{p-1} \nabla \cdot (v \nabla \psi) =-\int_{\RR^d} p(p-1) v^{p-1} \nabla v \cdot \nabla \psi=\int_{\RR^d} (p-1) v^{p} \Delta \psi.
$$
Symmetrizing the diffusive terms, we get
\begin{align*}
D_1&=-\int_{\RR^d} u(y,t)^{p-1} \Lambda^\alpha u(y,t) dy\\ &= -p \int_{\RR^d} u(y,t)^{p-1} \int_{\RR^d} \frac{u(y,t)-u(\eta,t)}{|y-\eta|^{d+\alpha}} d\eta dy \\
&=  -p \int_{\RR^d} u(\eta,t)^{p-1} \int_{\RR^d} \frac{u(\eta,t)-u(y,t)}{|\eta-y|^{d+\alpha}} d\eta dy \\
&=  p \int_{\RR^d} u(\eta,t)^{p-1} \int_{\RR^d} \frac{u(y,t)-u(\eta,t)}{|\eta-y|^{d+\alpha}} d\eta dy \\
&= -\frac{p}{2} \int_{\RR^d} \int_{\RR^d} \left(u(y,t)^{p-1}-u(\eta,t)^{p-1}\right) \frac{u(y,t)-u(\eta,t)}{|y-\eta|^{d+\alpha}} d\eta dy\\
& \leq 0  
\end{align*}
Following a similar procedure, 
\begin{align*}
D_2&=-\int_{\RR^d} v(y,t)^{p-1} \Lambda^\beta v(y,t) dy \\
&=-\frac{p}{2} \int_{\RR^d} \int_{\RR^d} \left(v(y,t)^{p-1}-v(\eta,t)^{p-1}\right) \frac{v(y,t)-v(\eta,t)}{|y-\eta|^{d+\beta}} d\eta dy\\
&\leq 0.
\end{align*} 
Thus, 
$$
\frac{d}{dt} \mathcal{F}_p[u,v]\leq T_1+T_2 =-(p-1)\int_{\RR^d}  (u^{p}-v^{p})(u-v)dx\leq 0,
$$ 
and we conclude
$$
\mathcal{F}_p[u,v]\leq \mathcal{F}_p[u_0,v_0].
$$

\textbf{Step 3: Boundedness in $L^\infty$} Due to the smoothness of $u(x,t)$ and $v(x,t)$ in space and time we have that 
$$
M_u(t):=\|u(t)\|_{L^\infty}=u(x_t,t),\;M_v(t):=\|v(t)\|_{L^\infty}=v(y_t,t)
$$ are Lipschitz. Thus, using Rademacher Theorem $M_u(t)$ and $M_v(t)$ are differentiable almost everywhere and (see \cite{AGM, cor2})
\begin{align*}
\frac{d}{dt} M_u(t) & =\pat u(x_t) \\
\frac{d}{dt} M_v(t)&=\pat v(y_t).
\end{align*}

Now we show the $\mathcal{F}_\infty[u,v]=\|u(t)\|_{L^\infty}+ \|v(t)\|_{L^\infty}$ is a Lyapunov functional:
\begin{align*}
\frac{d}{dt} \mathcal{F}_\infty[u,v] &= M_u(t) + M_v(t) \\
&= \pat u(x_t) + \pat v(y_t) \\
&= -\Lambda^\alpha u(x_t) - \Lambda^\beta v(y_t)-u(x_t)\Delta \psi(x_t)+ v(y_t)\Delta \psi(y_t) \\
&= -\Lambda^\alpha u(x_t) - \Lambda^\beta v(y_t)-u(x_t) [u(x_t)-v(x_t)]+ v(y_t)[u(y_t)-v(y_t)] \\
&\leq -\Lambda^\alpha u(x_t) - \Lambda^\beta v(y_t) -u(x_t)^2 +2u(x_t)v(y_t)-v(y_t)^2.
\end{align*}
Thus, using \eqref{eq:8}, we have that
\begin{align*}
-\Lambda^\alpha u(x_t) &= -\int_{\RR^d} \frac{u(x_t)-u(x_t-\eta)}{|\eta|^{d+\alpha}}d\eta \leq 0, \\
-\Lambda^\beta v(y_t) &= -\int_{\RR^d} \frac{v(y_t)-v(y_t-\eta)}{|\eta|^{d+\beta}}d\eta \leq 0,
\end{align*}
and
\begin{equation}\label{Lyapunofinfty}
\frac{d}{dt} \mathcal{F}_\infty[u,v] \leq -\Lambda^\alpha u(x_t) - \Lambda^\beta v(y_t)-(u(x_t)-v(y_t))^2 \leq0.
\end{equation}
So
$$
\mathcal{F}_\infty[u,v]\leq \mathcal{F}_\infty[u_0,v_0].
$$
\textbf{Step 4: Decay in $L^\infty$} Furthermore, we have the following lower bounds (see Lemma \ref{lemaaux3} and \cite{bae2015global})
\begin{align*}
\Lambda^\alpha u(x_t) &\geq \frac{c_{\alpha,d}u(x_t)\frac{\|u_0\|_{L^1}}{u(x_t)}}{\left(\left(\frac{2\|u_0\|_{L^1}}{u(x_t)}\right)^{1/d}\left(\frac{2}{\omega_d}\right)^{1/d}\right)^{d+\alpha}}\geq c(d,\alpha)\frac{u(x_t)^{1+\alpha/d}}{\|u_0\|^{\alpha/d}_{L^1}}\\
\Lambda^\beta v(y_t) &\geq \frac{c_{\beta,d}v(y_t)\frac{\|v_0\|_{L^1}}{v(y_t)}}{\left(\left(\frac{2\|v_0\|_{L^1}}{v(y_t)}\right)^{1/d}\left(\frac{2}{\omega_d}\right)^{1/d}\right)^{d+\beta}}\geq c(d,\beta)\frac{v(y_t)^{1+\beta/d}}{\|v_0\|^{\beta/d}_{L^1}},
\end{align*}
Thus, \eqref{Lyapunofinfty} can be sharpened and we get
\begin{equation}\label{Decayinfty1}
\frac{d}{dt} \mathcal{F}_\infty[u,v] \leq -c(d,\alpha)\frac{u(x_t)^{1+\alpha/d}}{\|u_0\|^{\alpha/d}_{L^1}}-c(d,\beta)\frac{v(y_t)^{1+\beta/d}}{\|v_0\|^{\beta/d}_{L^1}}.
\end{equation}

Fix $\gamma>0$. Then 
\begin{align*}
\left(u(x_t)+v(y_t)\right)^{1+\gamma}&\leq 2^{1+\gamma}\max\{u(x_t),v(y_t)\}^{1+\gamma}
\\
&\leq 2^{1+\gamma}\left(\max\{u(x_t),v(y_t)\}^{1+\gamma}+\min\{u(x_t),v(y_t)\}^
{1-\gamma+(\alpha+\beta)/d}\right).
\end{align*}
We define $\gamma$ as
\begin{equation}\label{gamma}
\gamma:=\left\{\begin{array}{ll}
\alpha/d\text{ if }\max\{u(x_t),v(y_t)\}=u(x_t)\\
\beta/d\text{ if }\max\{u(x_t),v(y_t)\}=v(y_t)
\end{array}\right..
\end{equation}
With this definition of $\gamma$, we have
\begin{align*}
\left(u(x_t)+v(y_t)\right)^{1+\gamma}&\leq 2^{1+\gamma}\max\{u(x_t),v(y_t)\}^{1+\gamma}
\\
&\leq 2^{1+\max\{\alpha,\beta\}/d}\left(u(x_t)^{1+\alpha/d}+v(y_t)^
{1+\beta/d}\right).
\end{align*}
Let us denote
\begin{equation}\label{Constant1}
C_{min}(\alpha,\beta,d,u_0,v_0):=\min\left\{\frac{c(d,\alpha)}{\|u_0\|^{\alpha/d}_{L^1}},\frac{c(d,\beta)}{\|v_0\|^{\beta/d}_{L^1}}\right\},
\end{equation}
then
$$
\frac{C_{min}}{2^{1+\max\{\alpha,\beta\}/d}}\left(u(x_t)+v(y_t)\right)^{1+\gamma}\leq c(d,\alpha)\frac{u(x_t)^{1+\alpha/d}}{\|u_0\|^{\alpha/d}_{L^1}}+c(d,\beta)\frac{v(y_t)^{1+\beta/d}}{\|v_0\|^{\beta/d}_{L^1}}.
$$
We obtain the inequality
$$
\frac{d}{dt}\mathcal{F}_\infty[u,v]\leq -\frac{C_{min}}{2^{1+\max\{\alpha,\beta\}/d}}\mathcal{F}_\infty[u,v]^{1+\gamma},
$$
where $\gamma$ is given by \eqref{gamma}. We obtain the following rate of decay,
$$
\mathcal{F}_\infty[u,v]\leq \frac{\mathcal{F}_\infty[u_0,v_0]}{\left(1+Kt\right)^{1/\gamma}}\leq \frac{\mathcal{F}_\infty[u_0,v_0]}{\left(1+Kt\right)^{d/\max\{\alpha,\beta\}}},
$$
where
$$
K=\min\left\{(\mathcal{F}_\infty[u_0,v_0])^{\alpha/d},(\mathcal{F}_\infty[u_0,v_0])^{\alpha/d}\right\}\frac{\min\{\alpha,\beta\}}{d}\frac{C_{min}}{2^{1+\max\{\alpha,\beta\}/d}}.
$$
As a consequence, we have
$$
\|u(t)\|_{L^\infty},\|v(t)\|_{L^\infty}\leq \frac{C_\infty}{\left(1+t\right)^{\frac{d}{\max\{\alpha,\beta\}}}}.
$$
\textbf{Step 5: Decay in $L^p$} Using interpolation and the conservation of mass, we obtain
$$
\|u(t)\|_{L^p}\leq \|u_0\|_{L^1}^{\frac{1}{p}}\frac{C_\infty^{1-\frac{1}{p}}}{\left(1+t\right)^{\frac{d}{\max\{\alpha,\beta\}}(1-\frac{1}{p})}},
$$
$$
\|v(t)\|_{L^p}\leq \|v_0\|_{L^1}^{\frac{1}{p}}\frac{C_\infty^{1-\frac{1}{p}}}{\left(1+t\right)^{\frac{d}{\max\{\alpha,\beta\}}(1-\frac{1}{p})}},
$$
$$
\mathcal{F}_p[u,v]\leq \left(\|u_0\|_{L^1}+\|v_0\|_{L^1}\right)\frac{C_\infty^{p-1}}{\left(1+t\right)^{\frac{d}{\max\{\alpha,\beta\}}(p-1)}}.
$$

\textbf{Step 6: Global existence} The global existence follows from the decay of $\|u\|_{L^\infty}+\|v\|_{L^\infty}$, energy estimates and a standard continuation argument (see \cite{AGM}).

\section{Proof of Theorem \ref{thdecay2}: Decay estimates in Sobolev spaces $H^s,\, 0<s<1$}\label{sec3}
\textbf{Step 1: Boundedness in $H^1$ ($d=1$)} 

First we deal with the one-dimensional case. We compute
\begin{align*}
\frac{1}{2}\frac{d}{dt}\|u\|_{\dot{H}^{1}}^2&=-\|u\|_{\dot{H}^{1+\alpha/2}}^2-\int_{\RR}\pax u \pax^2(u\pax\psi)\\
&\leq -\|u\|_{\dot{H}^{1+\alpha/2}}^2-\int_{\RR}\pax u (\pax^2 u\pax\psi+ u\pax(u-v)+2\pax u(u-v)),
\end{align*}
and
\begin{align*}
\frac{1}{2}\frac{d}{dt}\|v\|_{\dot{H}^{1}}^2&=-\|v\|_{\dot{H}^{1+\beta/2}}^2+\int_{\RR}\pax v \pax^2(v\pax\psi)\\
&\leq -\|v\|_{\dot{H}^{1+\beta/2}}^2+\int_{\RR}\pax v (\pax^2 v\pax\psi+ v\pax(u-v)+2\pax v(u-v)).
\end{align*}
Adding them together and using H\"{o}lder inequality, we have
\begin{align*}
\frac{d}{dt}\left(\|u\|_{\dot{H}^{1}}^2+\|v\|_{\dot{H}^{1}}^2\right)&=-2\|u\|_{\dot{H}^{1+\alpha/2}}^2-2\|v\|_{\dot{H}^{1+\beta/2}}^2\\
&\quad+C(\|\pax u\|_{L^2}^2+\|\pax v\|_{L^2}^2)(\|u\|_{L^\infty}+\|v\|_{L^\infty}).
\end{align*}
Using the interpolation inequality
\begin{equation}\label{interpoincare}
\|\Lambda^r f\|_{L^2}^2\leq \|f\|_{L^2}^2+\|f\|_{\dot{H}^{r+s}}^2,\;\forall\,r,s\geq0,
\end{equation}
we conclude that, for $t\geq T^*$ and $T^*<\infty$ large enough (see Theorem \ref{thdecay}), 
\begin{align*}
\frac{d}{dt}\left(\|u\|_{\dot{H}^{1}}^2+\|v\|_{\dot{H}^{1}}^2\right)&=-\|u\|_{\dot{H}^{1+\alpha/2}}^2-\|v\|_{\dot{H}^{1+\beta/2}}^2\\
&\quad+C(\|u\|_{L^2}^2+\|v\|_{L^2}^2)(\|u\|_{L^\infty}+\|v\|_{L^\infty}).
\end{align*}
Recalling that
$$
1<\frac{2}{\max\{\alpha,\beta\}}
$$ 
and using Theorem \ref{thdecay} to obtain that
$$
(\|u\|_{L^2}^2+\|v\|_{L^2}^2)(\|u\|_{L^\infty}+\|v\|_{L^\infty})\leq \frac{C}{(1+t)^{\frac{2}{\max\{\alpha,\beta\}}}},
$$
so
$$
\int_{T^*}^t(\|u\|_{L^2}^2+\|v\|_{L^2}^2)(\|u\|_{L^\infty}+\|v\|_{L^\infty})ds\leq C,
$$
we have that
$$
\|u(t)\|_{\dot{H}^1}^2+\|v(t)\|_{\dot{H}^1}^2+\int_{T^*}^t\|u\|_{\dot{H}^{1+\alpha/2}}^2+\|v\|_{\dot{H}^{1+\beta/2}}^2ds\leq C,\,\forall\,t\geq T^*.
$$
Standard energy estimates on the finite interval $[0,T^*]$ leads to
$$
\|u(t)\|_{\dot{H}^1}^2+\|v(t)\|_{\dot{H}^1}^2\leq C,\,\forall\,t\geq 0.
$$

\textbf{Step 2: Boundedness in $H^1$ ($d=2$, $d=3$)}
 
Assume now that $d=2$ or $d=3$. Testing the equation for $u$ against $\Lambda^{2} u$, we have
\begin{align*}
\frac{1}{2}\frac{d}{dt}\|u\|_{\dot{H}^{1}}^2&=-\|u\|_{\dot{H}^{1+\alpha/2}}^2-\int_{\RR^d}\Lambda u \Lambda(\nabla u \cdot \nabla \psi) dx-\int_{\RR^d}\Lambda^{2} u u (u-v) dx\\
&=-\|u\|_{\dot{H}^{1+\alpha/2}}^2-\int_{\RR^d}\Lambda u [\Lambda,\nabla \psi]\cdot\nabla u dx-\int_{\RR^d}\Lambda u \nabla \psi\cdot\nabla \Lambda u dx\\
&\quad-\int_{\RR^d}\Lambda u \Lambda(u(u-v))dx\\
&=-\|u\|_{\dot{H}^{1+\alpha/2}}^2-\int_{\RR^d}\Lambda u [\Lambda,\nabla \psi]\cdot\nabla u dx+\frac{1}{2}\int_{\RR^d}|\Lambda u|^2 (u-v) dx\\
&\quad-\int_{\RR^d}\Lambda u \Lambda(u(u-v)) dx.
\end{align*}

In the same way
\begin{align*}
\frac{1}{2}\frac{d}{dt}\|v\|_{\dot{H}^{1}}^2&=-\|v\|_{\dot{H}^{1+\beta/2}}^2+\int_{\RR^d}\Lambda v \Lambda (\nabla v \cdot \nabla \psi) dx+\int_{\RR^d}\Lambda^{2} v v (u-v) dx\\
&=-\|v\|_{\dot{H}^{1+\beta/2}}^2+\int_{\RR^d}\Lambda v [\Lambda,\nabla \psi]\cdot\nabla v dx+\int_{\RR^d}\Lambda v \nabla \psi\cdot\nabla \Lambda v dx\\
&\quad+\int_{\RR^d}\Lambda v \Lambda(v(u-v))dx\\
&=-\|v\|_{\dot{H}^{1+\beta/2}}^2+\int_{\RR^d}\Lambda v [\Lambda,\nabla \psi]\cdot\nabla v dx-\frac{1}{2}\int_{\RR^d}|\Lambda v|^2 (u-v) dx\\
&\quad+\int_{\RR^d}\Lambda v \Lambda(v(u-v)) dx.
\end{align*}

Recalling the Sobolev embedding
\begin{equation}\label{sobolevembed}
\|f\|_{L^{\frac{2d}{d-s}}}\leq C\|\Lambda^{s/2} f\|_{L^2},
\end{equation}
and Theorem \ref{thdecay} (using $d\geq2$, $\max\{\alpha,\beta\}<2$), we have a time $T^*<\infty$ such that, for $t\geq T^*$,
\begin{align}
\left|\int_{\RR^d}|\Lambda u|^2 (u-v) dx\right|&\leq \|\Lambda u\|^2_{L^{\frac{2d}{d-\alpha}}}\|u-v\|_{L^{d/\alpha}}\nonumber\\
&\leq C\|u\|^2_{\dot{H}^{1+\frac{\alpha}{2}}}\|u-v\|_{L^{d/\alpha}}\nonumber\\
&\leq \frac{1}{8}\|u\|^2_{\dot{H}^{1+\frac{\alpha}{2}}}\label{term4},\\
\left|\int_{\RR^d}|\Lambda v|^2 (u-v) dx\right|&\leq \|\Lambda v\|^2_{L^{\frac{2d}{d-\beta}}}\|u-v\|_{L^{d/\beta}}\nonumber\\
&\leq C\|v\|^2_{\dot{H}^{1+\frac{\beta}{2}}}\|u-v\|_{L^{d/\beta}}\nonumber\\
&\leq \frac{1}{8}\|v\|^2_{\dot{H}^{1+\frac{\beta}{2}}}\label{term3}.
\end{align}
Using the fractional Leibniz rule
$$
\|\Lambda^s(fg)\|_{L^p}\leq C(\|\Lambda^s f\|_{L^{p_1}}\|g\|_{L^{p_2}}+\|\Lambda^s g\|_{L^{p_3}}\|f\|_{L^{p_4}}),
$$
with
$$
\frac{1}{p}=\frac{1}{p_1}+\frac{1}{p_2}=\frac{1}{p_3}+\frac{1}{p_4},
$$
we have
\begin{multline}
\left|\int_{\RR^d}\Lambda u \Lambda(u(u-v)) dx\right|\\\leq C\|\Lambda u\|_{L^{\frac{2d}{d-\alpha}}}(\|\Lambda u\|_{L^2}\|u-v\|_{L^\frac{2d}{\alpha}}+\|u\|_{L^\frac{2d}{\alpha}}\|\Lambda(u-v)\|_{L^2})\\\leq C\|u\|_{\dot{H}^{1+\alpha/2}}(\|\Lambda u\|_{L^2}\|u-v\|_{L^\frac{2d}{\alpha}}+\|u\|_{L^\frac{2d}{\alpha}}\|\Lambda(u-v)\|_{L^2}),\label{term1}
\end{multline}
\begin{multline}
\left|\int_{\RR^d}\Lambda v \Lambda(v(u-v)) dx\right|\\\leq C\|\Lambda v\|_{L^{\frac{2d}{d-\beta}}}(\|\Lambda v\|_{L^2}\|u-v\|_{L^\frac{2d}{\beta}}+\|v\|_{L^\frac{2d}{\beta}}\|\Lambda(u-v)\|_{L^2})\\\leq C\|v\|_{\dot{H}^{1+\beta/2}}(\|\Lambda v\|_{L^2}\|u-v\|_{L^\frac{2d}{\beta}}+\|v\|_{L^\frac{2d}{\beta}}\|\Lambda(u-v)\|_{L^2}).\label{term2}
\end{multline}
Recalling the inequalities
\begin{equation}\label{riesz1}
\|\partial_{x_i}\partial_{x_j} f\|_{L^p}\leq C\|\Delta f\|_{L^p},\;\forall\, 1<p<\infty,
\end{equation}
\begin{equation}\label{riesz2}
\|\partial_{x_i}f\|_{L^p}\leq C\|\Lambda f\|_{L^p},\;\forall\, 1<p<\infty,
\end{equation}
and Lemma \ref{commutator2}, we have
\begin{equation}\label{comm1}
\|[\Lambda,\nabla\psi]\nabla u\|_{L^{\frac{2d}{d+\alpha}}}
\leq C\|\Delta\psi\|_{L^{\frac{2d}{\alpha}}}\|\Lambda u\|_{L^{2}},
\end{equation}
\begin{equation}\label{comm21}
\|[\Lambda,\nabla\psi]\nabla v\|_{L^{\frac{2d}{d+\beta}}}
\leq C\|\Delta \psi\|_{L^{\frac{2d}{\beta}}}\|\Lambda v\|_{L^{2}}.
\end{equation}
Thus, due to \eqref{comm1} and \eqref{comm21}, we have that
\begin{align}
\left|\int_{\RR^d}\Lambda u [\Lambda,\nabla \psi]\cdot\nabla u dx\right|&\leq C\|\Lambda u\|_{L^{\frac{2d}{d-\alpha}}}\|u-v\|_{L^{\frac{2d}{\alpha}}}\|\Lambda u\|_{L^{2}}\nonumber\\
&\leq C\|u\|_{\dot{H}^{1+\alpha/2}}\|u-v\|_{L^{\frac{2d}{\alpha}}}\|\Lambda u\|_{L^{2}} \label{term5},\\
\left|\int_{\RR^d}\Lambda u [\Lambda,\nabla \psi]\cdot\nabla v dx\right|&\leq C\|\Lambda v\|_{L^{\frac{2d}{d-\beta}}}\|u-v\|_{L^{\frac{2d}{\beta}}}\|\Lambda v\|_{L^{2}}\nonumber\\
&\leq C\|v\|_{\dot{H}^{1+\beta/2}}\|u-v\|_{L^{\frac{2d}{\beta}}}\|\Lambda v\|_{L^{2}} \label{term6}.
\end{align}

Collecting the terms \eqref{term4}, \eqref{term3}, \eqref{term1}, \eqref{term2}, \eqref{term5} and \eqref{term6} and using Young's inequality, we have that

\begin{align*}
\frac{d}{dt}\left(\|u\|_{\dot{H}^1}^2+\|v\|_{\dot{H}^1}^2\right)&\leq -\|u\|_{\dot{H}^{1+\alpha/2}}^2-\|v\|_{\dot{H}^{1+\beta/2}}^2+C\bigg{(}\|\Lambda u\|_{L^2}^2\|u-v\|_{L^\frac{2d}{\alpha}}^2\\
&\quad+\|u\|_{L^\frac{2d}{\alpha}}^2\|\Lambda(u-v)\|_{L^2}^2+\|\Lambda v\|_{L^2}^2\|u-v\|_{L^\frac{2d}{\beta}}^2\\
&\quad+\|v\|_{L^\frac{2d}{\beta}}^2\|\Lambda(u-v)\|_{L^2}^2\bigg{)}.
\end{align*}

Using the interpolation inequality \eqref{interpoincare}, we conclude that, for $t\geq T^*$ and $T^*<\infty$ large enough  (Theorem \ref{thdecay}), 
\begin{align*}
\frac{d}{dt}\left(\|u\|_{\dot{H}^1}^2+\|v\|_{\dot{H}^1}^2\right)&\leq C\bigg{(}\|u\|_{L^2}^2\|u-v\|_{L^\frac{2d}{\alpha}}^2+\|u\|_{L^\frac{2d}{\alpha}}^2\|u-v\|_{L^2}^2+\|v\|_{L^2}^2\|u-v\|_{L^\frac{2d}{\beta}}^2\\
&\quad+\|v\|_{L^\frac{2d}{\beta}}^2\|u-v\|_{L^2}^2\bigg{)}-\frac{1}{2}\left(\|u\|_{\dot{H}^{1+\alpha/2}}^2+\|v\|_{\dot{H}^{1+\beta/2}}^2\right).
\end{align*}
Another application of Theorem \ref{thdecay} leads to
\begin{multline*}
\frac{d}{dt}\left(\|u\|_{\dot{H}^1}^2+\|v\|_{\dot{H}^1}^2\right)+\frac{1}{2}\left(\|u\|_{\dot{H}^{1+\alpha/2}}^2+\|v\|_{\dot{H}^{1+\beta/2}}^2\right)\\
\leq C\left(\|v\|_{L^2}^2+\|u\|_{L^2}^2\right)\left(\|v\|_{L^\frac{2d}{\alpha}}^2+\|u\|_{L^\frac{2d}{\alpha}}^2+\|u\|_{L^\frac{2d}{\beta}}^2+\|v\|_{L^\frac{2d}{\beta}}^2\right)\\
\leq \frac{C}{(1+t)^{\frac{d}{\max\{\alpha,\beta\}}}}\left(\frac{1}{\left(1+t\right)^{\frac{2d(1-\frac{\alpha}{2d})}{\max\{\alpha,\beta\}}}}+\frac{1}{\left(1+t\right)^{\frac{2d(1-\frac{\beta}{2d})}{\max\{\alpha,\beta\}}}}\right)\\
\leq C\left(\frac{1}{\left(1+t\right)^{\frac{3d-\max\{\alpha,\beta\}}{\max\{\alpha,\beta\}}}}+\frac{1}{\left(1+t\right)^{\frac{3d-\beta}{\max\{\alpha,\beta\}}}}\right)\\
\leq \frac{C}{\left(1+t\right)^{\frac{3d-\max\{\alpha,\beta\}}{\max\{\alpha,\beta\}}}}.
\end{multline*}
Using Theorem \ref{thdecay} with $\max\{\alpha,\beta\}<2$ we obtain the inequality
\begin{equation}\label{exponent}
1<\frac{2}{\max\{\alpha,\beta\}}\leq \frac{d}{\max\{\alpha,\beta\}},
\end{equation}
thus, we have that
\begin{equation}\label{exponent2}
2< \frac{3d}{\max\{\alpha,\beta\}}-1.
\end{equation}
Integrating in time, we obtain

$$
\|u(t)\|_{\dot{H}^1}^2+\|v(t)\|_{\dot{H}^1}^2+\frac{1}{2}\int_{T^*}^t\|u\|_{\dot{H}^{1+\alpha/2}}^2+\|v\|_{\dot{H}^{1+\beta/2}}^2ds\leq C,\,\forall\,t\geq T^*,
$$

Taking then the maximum of the norms on the finite interval $[0,T^*]$, we obtain
\begin{equation}\label{boundH1}
\|u(t)\|_{\dot{H}^1}^2+\|v(t)\|_{\dot{H}^1}^2\leq C,\,\forall\,t\geq 0,
\end{equation}

\textbf{Step 3: Decay in $H^{s}$} 

Sobolev interpolation
\begin{equation}\label{interpolationsob}
\|f\|_{\dot{H}^s}\leq C\|f\|_{L^2}^{\frac{r-s}{r}}\|f\|_{\dot{H}^{r}}^{\frac{s}{r}},
\end{equation}
(with $r=1$) gives us the following decay in the intermediate spaces $\dot{H}^{s}$ for every $0\leq s<1$
\begin{equation}\label{boundH1gamma}
\|u(t)\|_{\dot{H}^{s}}+\|v(t)\|_{\dot{H}^{s}}\leq \frac{C}{(1+t)^{\frac{d}{\max\{\alpha,\beta\}}\frac{1-s}{2}}},\,\forall\,t\geq 0,
\end{equation}

\section{Proof of Theorem \ref{thdecay3}: Decay estimates in Sobolev spaces $H^s,\, 0\leq s<2$}\label{sec4}
 
\textbf{Step 1: Boundedness in $H^2$} Testing against $(-\Delta)^2 u$, we have
\begin{align*}
\frac{1}{2}\frac{d}{dt}\|u\|_{\dot{H}^{2}}^2&=-\|u\|_{\dot{H}^{2+\alpha/2}}^2-\int_{\RR^d}\Delta u \Delta(\nabla u \cdot \nabla \psi) dx-\int_{\RR^d}\Delta u \Delta(u (u-v)) dx\\
&=-\|u\|_{\dot{H}^{2+\alpha/2}}^2-\int_{\RR^d}\Delta u [\Delta,\nabla \psi]\cdot\nabla u dx-\int_{\RR^d}\Delta u \nabla \psi\cdot\nabla \Delta u dx\\
&\quad-\int_{\RR^d}\Delta u \Delta(u(u-v)) dx\\
&=-\|u\|_{\dot{H}^{2+\alpha/2}}^2-\int_{\RR^d}\Delta u [\Delta,\nabla \psi]\cdot\nabla u dx+\frac{1}{2}\int_{\RR^d}|\Delta u|^2 (u-v) dx\\
&\quad-\int_{\RR^d}\Delta u \left(\Delta u(u-v) + u\Delta(u-v)+2\nabla u \cdot \nabla (u-v)\right) dx\\
&= -\|u\|_{\dot{H}^{2+\alpha/2}}^2-\int_{\RR^d}\Delta u [\Delta,\nabla \psi]\cdot\nabla u dx-\frac{1}{2}\int_{\RR^d}|\Delta u|^2 (u-v) dx\\
&\quad-\int_{\RR^d}\Delta u \left( u\Delta(u-v)+2\nabla u \cdot \nabla (u-v)\right) dx\\
&\leq -\|u\|_{\dot{H}^{2+\alpha/2}}^2-\int_{\RR^d}\Delta u [\Delta,\nabla \psi]\cdot\nabla u dx+\frac{1}{2}\int_{\RR^d}|\Delta u|^2 v dx\\
&\quad+\int_{\RR^d}\Delta u \left( u\Delta v+2\nabla u \cdot \nabla (u-v)\right) dx,
\end{align*}
In the same way
\begin{align*}
\frac{1}{2}\frac{d}{dt}\|v\|_{\dot{H}^{2}}^2&=-\|v\|_{\dot{H}^{2+\beta/2}}^2+\int_{\RR^d}\Delta v \Delta(\nabla v \cdot \nabla \psi) dx+\int_{\RR^d}\Delta v \Delta(v (u-v)) dx\\
&\leq -\|v\|_{\dot{H}^{2+\beta/2}}^2+\int_{\RR^d}\Delta v [\Delta,\nabla \psi]\cdot\nabla v dx+\frac{1}{2}\int_{\RR^d}|\Delta v|^2 u dx\\
&\quad+\int_{\RR^d}\Delta v \left(v\Delta u+2\nabla v\cdot \nabla(u-v) \right)dx.
\end{align*}
We collect this estimates and use H\"{o}lder inequality to obtain

\begin{align*}
\frac{1}{2}\frac{d}{dt}\left(\|u\|_{\dot{H}^{2}}^2+\|v\|_{\dot{H}^{2}}^2\right)&\leq -\|u\|_{\dot{H}^{2+\alpha/2}}^2+\|\Delta u\|_{L^{\frac{2d}{d-\alpha}}}\|[\Delta,\nabla \psi]\cdot\nabla u\|_{L^{\frac{2d}{d+\alpha}}}\\
&\quad+\frac{1}{2}\|\Delta u\|_{L^{\frac{2d}{d-\alpha}}}^2\|v\|_{L^{\frac{d}{\alpha}}}+\frac{1}{2}\|\Delta v\|_{L^{\frac{2d}{d-\beta}}}^2\|u\|_{L^{\frac{d}{\beta}}}\\
&\quad+2\|\Delta u\|_{L^{\frac{2d}{d-\alpha}}}\|\nabla u\|_{L^{\frac{4d}{d+\alpha}}}\|\nabla (u-v)\|_{L^{\frac{4d}{d+\alpha}}}\\
&\quad +2\|\Delta v\|_{L^{\frac{2d}{d-\beta}}}\|\nabla v\|_{L^{\frac{4d}{d+\beta}}}\|\nabla (u-v)\|_{L^{\frac{4d}{d+\beta}}}\\
&\quad-\|v\|_{\dot{H}^{2+\beta/2}}^2+\|\Delta v\|_{L^{\frac{2d}{d-\beta}}} \|[\Delta,\nabla \psi]\cdot\nabla v \|_{L^{\frac{2d}{d+\beta}}}\\
&\quad+\|\Delta u\|_{L^{\frac{2d}{d-\alpha}}}\|\Delta v\|_{L^{\frac{2d}{d-\beta}}}\|u+v\|_{L^{\frac{2d}{\alpha+\beta}}}.
\end{align*}
Due the Sobolev embedding \eqref{sobolevembed}, we have that, for $t\geq T^*$ and $T^*<\infty$ large enough, the previous inequality simplifies to

\begin{align*}
\frac{d}{dt}\left(\|u\|_{\dot{H}^{2}}^2+\|v\|_{\dot{H}^{2}}^2\right)&\leq -2\|u\|_{\dot{H}^{2+\alpha/2}}^2+C\|u\|_{\dot{H}^{2+\alpha/2}}\|[\Delta,\nabla \psi]\cdot\nabla u\|_{L^{\frac{2d}{d+\alpha}}}\\
&\quad+C\|u\|_{\dot{H}^{2+\alpha/2}}\|\nabla u\|_{L^{\frac{4d}{d+\alpha}}}\|\nabla (u-v)\|_{L^{\frac{4d}{d+\alpha}}}\\
&\quad +C\|v\|_{\dot{H}^{2+\beta/2}}\|\nabla v\|_{L^{\frac{4d}{d+\beta}}}\|\nabla (u-v)\|_{L^{\frac{4d}{d+\beta}}}\\
&\quad-2\|v\|_{\dot{H}^{2+\beta/2}}^2+C\|v\|_{\dot{H}^{2+\beta/2}} \|[\Delta,\nabla \psi]\cdot\nabla v \|_{L^{\frac{2d}{d+\beta}}}.
\end{align*}

Lemma \ref{commutator2} together with \eqref{riesz1} and \eqref{riesz2} give us the following estimates
\begin{multline}\label{comm}
\|[\Delta,\nabla\psi]\nabla u\|_{L^{\frac{2d}{d+\alpha}}}\\
\leq C\left(\|u-v\|_{L^{\frac{2d}{\alpha}}}\|\Delta u\|_{L^{2}}+\|\nabla(u-v)\|_{L^{\frac{4d}{d+\alpha}}}\|\nabla u\|_{L^{\frac{4d}{d+\alpha}}}\right),
\end{multline}
\begin{multline}\label{comm2}
\|[\Delta,\nabla\psi]\nabla v\|_{L^{\frac{2d}{d+\beta}}}\\
\leq C\left(\|u-v\|_{L^{\frac{2d}{\beta}}}\|\Delta v\|_{L^{2}}+\|\nabla(u-v)\|_{L^{\frac{4d}{d+\beta}}}\|\nabla v\|_{L^{\frac{4d}{d+\beta}}}\right).
\end{multline}

Consequently, due to the interpolation inequality \eqref{interpoincare} with $r=2$, we can further simplify and get
\begin{align*}
\frac{d}{dt}\left(\|u\|_{\dot{H}^{2}}^2+\|v\|_{\dot{H}^{2}}^2\right)&\leq -\|u\|_{\dot{H}^{2+\alpha/2}}^2-\|v\|_{\dot{H}^{2+\beta/2}}^2\\
&\quad+C\|u-v\|_{L^{\frac{2d}{\beta}}}^2\|v\|_{L^{2}}^2+C\|u-v\|_{L^{\frac{2d}{\alpha}}}^2\|u\|_{L^{2}}^2\\
&\quad+C\|u\|_{\dot{H}^{2+\alpha/2}}\|\nabla u\|_{L^{\frac{4d}{d+\alpha}}}\|\nabla (u-v)\|_{L^{\frac{4d}{d+\alpha}}}\\
&\quad +C\|v\|_{\dot{H}^{2+\beta/2}}\|\nabla v\|_{L^{\frac{4d}{d+\beta}}}\|\nabla (u-v)\|_{L^{\frac{4d}{d+\beta}}}.
\end{align*}

Using the Sobolev embedding
\begin{align*}
\|\nabla f\|_{L^{\frac{4d}{d+s}}}&\leq C\|f\|_{\dot{H}^{1+\frac{d-s}{4}}}\leq C\|\Lambda^{1-\frac{s}{4}}f\|_{\dot{H}^{\frac{d}{4}}},
\end{align*}
and the interpolation inequality \eqref{interpolationsob} we have that
\begin{align*}
I_1&=\|u\|_{\dot{H}^{2+\alpha/2}}\|\nabla u\|_{L^{\frac{4d}{d+\alpha}}}\|\nabla (u-v)\|_{L^{\frac{4d}{d+\alpha}}}\\
&\leq C\|u\|_{\dot{H}^{2+\alpha/2}}\|\Lambda^{1-\frac{\alpha}{4}}u\|_{\dot{H}^{\frac{d}{4}}}\left(\|\Lambda^{1-\frac{\alpha}{4}}u\|_{\dot{H}^{\frac{d}{4}}}+\|\Lambda^{1-\frac{\alpha}{4}}v\|_{\dot{H}^{\frac{d}{4}}}\right)\\
&\leq C\|u\|_{\dot{H}^{2+\alpha/2}}^{1+\frac{d}{4+3\alpha}}\|\Lambda^{1-\frac{\alpha}{4}}u\|_{L^2}^{1-\frac{d}{4+3\alpha}}\left(\|\Lambda^{1-\frac{\alpha}{4}}u\|_{\dot{H}^{\frac{d}{4}}}+\|\Lambda^{1-\frac{\alpha}{4}}v\|_{\dot{H}^{\frac{d}{4}}}\right)\\
&\leq C\|u\|_{\dot{H}^{2+\alpha/2}}^{1+\frac{d}{4+3\alpha}}\|u\|_{\dot{H}^{1-\frac{\alpha}{4}}}^{1-\frac{d}{4+3\alpha}}\left(\|u\|_{\dot{H}^{1-\frac{\alpha}{4}}}^{1-\frac{d}{4+3\alpha}}\|u\|_{\dot{H}^{2+\alpha/2}}^{\frac{d}{4+3\alpha}}+\|v\|_{\dot{H}^{2+\beta/2}}^{\frac{d}{4+\alpha+2\beta}}\|v\|_{\dot{H}^{1-\frac{\alpha}{4}}}^{1-\frac{d}{4+\alpha+2\beta}}\right),
\end{align*}
and
\begin{align*}
I_2&=\|v\|_{\dot{H}^{2+\beta/2}}\|\nabla v\|_{L^{\frac{4d}{d+\beta}}}\|\nabla (u-v)\|_{L^{\frac{4d}{d+\beta}}}\\
&\leq C\|v\|_{\dot{H}^{2+\beta/2}}\|\Lambda^{1-\frac{\beta}{4}}v\|_{\dot{H}^{\frac{d}{4}}}\left(\|\Lambda^{1-\frac{\beta}{4}}u\|_{\dot{H}^{\frac{d}{4}}}+\|\Lambda^{1-\frac{\beta}{4}}v\|_{\dot{H}^{\frac{d}{4}}}\right)\\
&\leq C\|v\|_{\dot{H}^{2+\beta/2}}^{1+\frac{d}{4+3\beta}}\|v\|_{\dot{H}^{1-\frac{\beta}{4}}}^{1-\frac{d}{4+3\beta}}\left(\|u\|_{\dot{H}^{1-\frac{\alpha}{4}}}^{1-\frac{d}{4+\alpha+2\beta}}\|u\|_{\dot{H}^{2+\alpha/2}}^{\frac{d}{4+\alpha+2\beta}}+\|v\|_{\dot{H}^{2+\beta/2}}^{\frac{d}{4+3\beta}}\|v\|_{\dot{H}^{1-\frac{\beta}{4}}}^{1-\frac{d}{4+3\beta}}\right).
\end{align*}
We write
$$
I_1+I_2=J_1+J_2+J_3+J_4,
$$
where
\begin{align*}
J_1&=C\|u\|_{\dot{H}^{2+\alpha/2}}^{1+\frac{2d}{4+3\alpha}}\|u\|_{\dot{H}^{1-\frac{\alpha}{4}}}^{2-\frac{2d}{4+3\alpha}}\\
J_2&=C\|v\|_{\dot{H}^{2+\beta/2}}^{1+\frac{2d}{4+3\beta}}\|v\|_{\dot{H}^{1-\frac{\beta}{4}}}^{2-\frac{2d}{4+3\beta}}\\
J_3&=C\|u\|_{\dot{H}^{2+\alpha/2}}^{1+\frac{d}{4+3\alpha}}\|u\|_{\dot{H}^{1-\frac{\alpha}{4}}}^{1-\frac{d}{4+3\alpha}}\|v\|_{\dot{H}^{2+\beta/2}}^{\frac{d}{4+\alpha+2\beta}}\|v\|_{\dot{H}^{1-\frac{\alpha}{4}}}^{1-\frac{d}{4+\alpha+2\beta}}\\
J_4&=C\|v\|_{\dot{H}^{2+\beta/2}}^{1+\frac{d}{4+3\beta}}\|v\|_{\dot{H}^{1-\frac{\beta}{4}}}^{1-\frac{d}{4+3\beta}}\|u\|_{\dot{H}^{1-\frac{\alpha}{4}}}^{1-\frac{d}{4+\alpha+2\beta}}\|u\|_{\dot{H}^{2+\alpha/2}}^{\frac{d}{4+\alpha+2\beta}}\\
\end{align*}
Using hypothesis \eqref{condition}, so that
$$
\frac{2d}{4+3\min\{\alpha,\beta\}}< 1,
$$
we can apply Young's inequality with 
$$
p=2-\frac{4d}{4+3\alpha+2d},\;q=2+\frac{4d}{4+3\alpha-2d}
$$
and, recalling Theorem \ref{thdecay2}, we obtain that
\begin{align*}
J_1&\leq \frac{1}{4}\|u\|_{\dot{H}^{2+\alpha/2}}^{2}+C\|u\|_{\dot{H}^{1-\frac{\alpha}{4}}}^{(2-\frac{2d}{4+3\alpha})q}\\
&\leq \frac{1}{4}\|u\|_{\dot{H}^{2+\alpha/2}}^{2}+\frac{C}{(1+t)^{\theta_1}},
\end{align*}
where
\begin{equation}\label{theta1}
\theta_1=\frac{d}{\max\{\alpha,\beta\}}\frac{\alpha}{8}\left(2-\frac{2d}{4+3\alpha}\right)q=\frac{d}{\max\{\alpha,\beta\}}\frac{\alpha}{4}\frac{8+6\alpha-2d}{4+3\alpha-2d}.
\end{equation}
We need to have $\theta>1$. Then, in the case where $\beta=\max\{\alpha,\beta\}$ and $\alpha\ll1$, the previous exponent may be arbitrarily small. However, in the case where \eqref{condition2} holds, we have that
$$
\theta_1\geq\frac{\min\{\alpha,\beta\}}{\max\{\alpha,\beta\}}\frac{d}{4}\left(2+\frac{2d}{4+3\max\{\alpha,\beta\}-2d}\right)>1.
$$
Applying Young's inequality now with
$$
p=2-\frac{4d}{4+3\beta+2d},\;q=2+\frac{4d}{4+3\beta-2d},
$$
we have that
\begin{align*}
J_2&\leq \frac{1}{4}\|v\|_{\dot{H}^{2+\beta/2}}^{2}+C\|v\|_{\dot{H}^{1-\frac{\beta}{4}}}^{(2-\frac{2d}{4+3\beta})q}\\
&\leq \frac{1}{4}\|v\|_{\dot{H}^{2+\alpha/2}}^{2}+\frac{C}{(1+t)^{\theta_2}},
\end{align*}
where
\begin{equation}\label{theta2}
\theta_2=\frac{d}{\max\{\alpha,\beta\}}\frac{\beta}{8}\left(2-\frac{2d}{4+3\beta}\right)q=\frac{d}{\max\{\alpha,\beta\}}\frac{\beta}{4}\frac{8+6\beta-2d}{4+3\beta-2d}.
\end{equation}
Thus, using hypothesis \eqref{condition2}, we have that
$$
\theta_2>1.
$$
Using again Young's inequality with
$$
p=2-\frac{2d}{4+3\alpha+d},\;q=2+\frac{2d}{4+3\alpha-d}
$$
\begin{align*}
J_3&\leq
\frac{1}{4}\|u\|_{\dot{H}^{2+\alpha/2}}^{2}+C\|u\|_{\dot{H}^{1-\frac{\alpha}{4}}}^{(1-\frac{d}{4+3\alpha})q}\|v\|_{\dot{H}^{2+\beta/2}}^{\lambda}\|v\|_{\dot{H}^{1-\frac{\alpha}{4}}}^{(1-\frac{d}{4+\alpha+2\beta})q}.
\end{align*}
Due to hypothesis \eqref{condition3} the exponent is
\begin{align*}
\lambda&=\frac{d}{4+\alpha+2\beta}q\\
&=\frac{d}{4+3\alpha-d}\left(2+\frac{4(\alpha-\beta)}{4+\alpha+2\beta}\right)\\
&\leq \frac{d}{4+3\min\{\alpha,\beta\}-d}\left(2+\frac{4|\alpha-\beta|}{4+3\min\{\alpha,\beta\}}\right)\\
&\leq 2.
\end{align*}
Assume that $\lambda<2$ (if $\Gamma=2$, we can finish with $J_3$ straightforwardly by waiting for a large enough time and applying Theorem \ref{thdecay2}), thus, we can apply Young's inequality again 
$$
P=\frac{2}{\lambda},\;Q=\frac{2}{2-\lambda}
$$
and obtain
\begin{align*}
J_3&\leq
\frac{1}{4}\|u\|_{\dot{H}^{2+\alpha/2}}^{2}+\frac{1}{4}\|v\|_{\dot{H}^{2+\beta/2}}^{2}+C\|u\|_{\dot{H}^{1-\frac{\alpha}{4}}}^{(1-\frac{d}{4+3\alpha})qQ}\|v\|_{\dot{H}^{1-\frac{\alpha}{4}}}^{(1-\frac{d}{4+\alpha+2\beta})qQ}\\
&\leq\frac{1}{4}\|u\|_{\dot{H}^{2+\alpha/2}}^{2}+\frac{1}{4}\|v\|_{\dot{H}^{2+\beta/2}}^{2}+\frac{C}{(1+t)^{\theta_3}}.
\end{align*}
Notice that the condition
$$
\left(1-\frac{d}{4+3\alpha}\right)qQ>2
$$
is implied by the stricter condition
$$
\left(1-\frac{d}{4+3\alpha}\right)Q>1,
$$
or, equivalently,
$$
\frac{2}{4+3\alpha}<\frac{q}{4+\alpha+2\beta}=\frac{2+\frac{2d}{4+3\alpha-d}}{4+\alpha+2\beta}.
$$
A further computation shows that this latter condition is implied by hypothesis \eqref{condition4}
$$
\frac{4+\alpha+2\beta}{4+3\alpha}\leq \frac{4+3\max\{\alpha,\beta\}}{4+3\min\{\alpha,\beta\}} <1+\frac{d}{4+3\max\{\alpha,\beta\}-d}.
$$
Then, using Theorem \ref{thdecay2}, the integrability condition $\theta_3>1$
is implied by
$$
\frac{d}{\max\{\alpha,\beta\}}\frac{\alpha}{2}\geq1,
$$
and hypothesis \eqref{condition5}. 

The term $J_4$ is akin to $J_3$ and can be handled similarly. Then, we obtain
$$
\frac{d}{dt}\left(\|u\|_{\dot{H}^{2}}^2+\|v\|_{\dot{H}^{2}}^2\right)\leq \frac{C}{(1+t)^{\Theta}},\,\forall\,t\geq T^*,
$$
and $\Theta>1$. Thus,
$$
\|u(t)\|_{\dot{H}^{2}}^2+\|v(t)\|_{\dot{H}^{2}}^2\leq C,\,\forall\,t\geq 0.
$$

\textbf{Step 2: Decay in $H^s$} Using \eqref{interpolationsob}, we have
\begin{equation}\label{boundH2gamma}
\|u(t)\|_{\dot{H}^{s}}+\|v(t)\|_{\dot{H}^{s}}\leq \frac{C}{(1+t)^{\frac{d}{\max\{\alpha,\beta\}}\frac{2-s}{4}}},\,\forall\,t\geq 0,
\end{equation}

\section{Proof of Theorem \ref{thdecay4}: Decay estimates in Sobolev spaces $H^s$, $s>2$}\label{sec5}
Let us fix $\delta=2-d/2$. Testing against $\Lambda^{2s}u$ we have the following estimate
\begin{align*}
\frac{1}{2}\frac{d}{dt}\|u\|_{\dot{H}^{s}}^2&=-\|u\|_{\dot{H}^{s+\alpha/2}}^2-\int_{\RR^d}\Lambda^{s}u \Lambda^s\nabla\cdot (u \nabla \psi) dx\\
&=-\|u\|_{\dot{H}^{s+\alpha/2}}^2-\int_{\RR^d}\Lambda^{s}u [\Lambda^s\nabla,\nabla \psi]u dx-\int_{\RR^d}\Lambda^{s}u \Lambda^s\nabla u \cdot\nabla \psi dx\\
&=-\|u\|_{\dot{H}^{s+\alpha/2}}^2-\int_{\RR^d}\Lambda^{s}u [\Lambda^s\nabla,\nabla \psi]u dx+\frac{1}{2}\int_{\RR^d}|\Lambda^s u|^2 (u-v) dx.
\end{align*}
Similarly
\begin{align*}
\frac{1}{2}\frac{d}{dt}\|v\|_{\dot{H}^{s}}^2&=-\|v\|_{\dot{H}^{s+\beta/2}}^2+\int_{\RR^d}\Lambda^{s}v \Lambda^s\nabla\cdot (v \nabla \psi) dx\\
&=-\|v\|_{\dot{H}^{s+\beta/2}}^2+\int_{\RR^d}\Lambda^{s}v [\Lambda^s\nabla,\nabla \psi]v dx+\int_{\RR^d}\Lambda^{s}v \Lambda^s\nabla v \cdot\nabla \psi dx\\
&=-\|v\|_{\dot{H}^{s+\beta/2}}^2+\int_{\RR^d}\Lambda^{s}v [\Lambda^s\nabla,\nabla \psi]v dx-\frac{1}{2}\int_{\RR^d}|\Lambda^s v|^2 (u-v) dx\\
\end{align*}
Using Lemma \ref{commutator}, we have that
\begin{align*}
\|[\Lambda^s\nabla,\nabla \psi]f\|_{L^2}&\leq C\left(\|\Lambda^s f\|_{L^2}\|\widehat{\Lambda \nabla\psi}\|_{L^1}+\|\Lambda^{s+1} \nabla\psi\|_{L^2}\|\widehat{f}\|_{L^1}\right)\\
&\leq C\left(\|\Lambda^s f\|_{L^2}\|\widehat{\Delta \psi}\|_{L^1}+\|\Lambda^{s} \Delta\psi\|_{L^2}\|\widehat{f}\|_{L^1}\right)\\
&\leq C\left(\|\Lambda^s f\|_{L^2}\left(\|\hat{u}\|_{L^1}+\|\hat{v}\|_{L^1}\right)\right.\\
&\quad\left.+\left(\|\Lambda^{s} u\|_{L^2}+\|\Lambda^{s} v\|_{L^2}\right)\|\widehat{f}\|_{L^1}\right).
\end{align*}
That means that
\begin{align*}
\frac{d}{dt}\left(\|u\|_{\dot{H}^{s}}^2+\|v\|_{\dot{H}^{s}}^2\right)&\leq -2\|u\|_{\dot{H}^{s+\alpha/2}}^2-2\|v\|_{\dot{H}^{s+\beta/2}}^2\\
&\quad +C\left(\| u\|_{\dot{H}^{s}}^2\left(\|\hat{u}\|_{L^1}+\|\hat{v}\|_{L^1}\right)\right.\\
&\quad+\left(\|u\|_{\dot{H}^{s}}^2+\|u\|_{\dot{H}^{s}}\|v\|_{\dot{H}^{s}}\right)\|\widehat{u}\|_{L^1}\\
&\quad +\|v\|_{\dot{H}^{s}}^2\left(\|\hat{u}\|_{L^1}+\|\hat{v}\|_{L^1}\right)\\
&\quad\left.+\left(\|u\|_{\dot{H}^{s}}\|v\|_{\dot{H}^{s}}+\|v\|_{\dot{H}^{s}}^2\right)\|\widehat{v}\|_{L^1}\right),
\end{align*}
where we have used the inequality
$$
\|f\|_{L^\infty}\leq \|\hat{f}\|_{L^1}.
$$
We obtain that
\begin{align}
\frac{1}{2}\left(\frac{d}{dt}\|u\|_{\dot{H}^{s}}^2+\frac{d}{dt}\|v\|_{\dot{H}^{s}}^2\right)&\leq -\|u\|_{\dot{H}^{s+\alpha/2}}^2-\|v\|_{\dot{H}^{s+\beta/2}}^2\nonumber\\
&\quad +C\left(\| v\|_{\dot{H}^{s}}^2+\|u\|_{\dot{H}^{s}}^2\right)\left(\|\hat{u}\|_{L^1}+\|\hat{v}\|_{L^1}\right).\label{eq:a}
\end{align}
Using Lemma \ref{lemmainterpolation} (inequality \eqref{eq:10b}) and Theorem \ref{thdecay3}, we have that
\begin{align*}
\|\hat{u}\|_{L^1(\RR^d)}+\|\hat{v}\|_{L^1(\RR^d)}&\leq C(1+t)^{-\frac{\delta}{2}\frac{d/2}{\max\{\alpha,\beta\}}}\left(\|u\|_{\dot{H}^{2}(\RR^d)}^{\frac{d/2}{2}}+\|v\|_{\dot{H}^{2}(\RR^d)}^{\frac{d/2}{2}}\right)\\
&\leq C(1+t)^{-\frac{\delta}{2}\frac{d/2}{\max\{\alpha,\beta\}}}.
\end{align*}
Thus, waiting for a large enough time $T^*$ and using \eqref{interpoincare} and 
$$
\int_{T^*}^t\|u\|_{L^2}^2+\|v\|_{L^2}^2=C<\infty, 
$$
we conclude
$$
\|u\|_{\dot{H}^{s}}^2+\|v\|_{\dot{H}^{s}}^2\leq C,\,\forall\,t\geq T^*.
$$
Considering the \emph{a priori} estimates in the finite interval $[0,T^*]$ we conclude
\begin{equation}\label{boundHs}
\|u\|_{\dot{H}^{s}}^2+\|v\|_{\dot{H}^{s}}^2\leq C,\,\forall\,t\geq 0.
\end{equation}
Now the decay follows from interpolation \eqref{interpolationsob}.

\section{Proof of Proposition \ref{thasym}: Asymptotic profile}\label{sec6}
\textbf{Step 1: Decay of the potential $\psi$} This step is similar to the one in \cite{yamamoto2015asymptotic}. We have
$$
\psi=\Delta^{-1}(u-v),
$$
so, using Theorem \ref{thdecay}, we have that
\begin{align*}
\nabla\psi&=C_d\int_{\RR^d}\frac{x_i-y_i}{|x-y|^d}(u(y)-v(y))dy\\
&=C_d\left(\int_{|y|\leq(1+t)^r}+\int_{|y|>(1+t)^r}\right)\frac{x_i-y_i}{|x-y|^d}(u(y)-v(y))dy\\
&\leq C\left(\|u\|_{L^\infty}+\|v\|_{L^\infty}\right)(1+t)^r+\left(\|u\|_{L^1}+\|v\|_{L^1}\right)(1+t)^{r(d-1)}\\
&\leq C\left(1+t\right)^{r-\frac{d}{\max\{\alpha,\beta\}}}+C(1+t)^{-r(d-1)}.
\end{align*}
We choose $r=\frac{1}{\max\{\alpha,\beta\}}$, thus, we obtain
\begin{equation}\label{potential}
\|\nabla\psi\|_{L^\infty}\leq \frac{C}{(1+t)^{\frac{d-1}{\max\{\alpha,\beta\}}}}.
\end{equation}

\textbf{Step 2: Mild solution} Using Duhamel's principle, the mild solutions are given by
\begin{align*}
u(t)-e^{-t\Lambda^{\alpha}}u_0&=-\int_0^te^{-(t-s)\Lambda^{\alpha}}\nabla\cdot(u(s)\nabla\psi(s))ds\\
v(t)-e^{-t\Lambda^{\beta}}v_0&=\int_0^te^{-(t-s)\Lambda^{\beta}}\nabla\cdot(v(s)\nabla\psi(s))ds.
\end{align*}

\textbf{Step 3: Estimate on the difference} Using the hypercontractive inequality
$$
\|e^{-t\Lambda^\alpha}h\|_{L^2}\leq Ct^{-\frac{d}{2\alpha}}\|h\|_{L^1}
$$
we have that
\begin{align*}
\|u(t)-e^{-t\Lambda^{\alpha}}u_0\|_{L^2}&=\left\|\int_0^te^{-(t-s)\Lambda^{\alpha}}\nabla\cdot(u(s)\nabla\psi(s))ds\right\|_{L^2}\\
&\leq C\int_0^{t/2}\frac{f(s)}{(t-s)^{\frac{d}{2\alpha}}}ds+C\int_{t/2}^t g(s)ds,
\end{align*}
where the forcing are
\begin{align*}
f(s)&=\left\|u(s)(u(s)-v(s))\right\|_{L^1}\\
&\leq \frac{C}{(1+s)^{\frac{d}{\max\{\alpha,\beta\}}}}
\end{align*}
and
\begin{align*}
g(s)&=\left\|\nabla u(s)\cdot \nabla\psi(s)\right\|_{L^2}+\left\|u(s)(u(s)-v(s))\right\|_{L^2}\\
&\leq \frac{C}{(1+s)^{\frac{d-1}{\max\{\alpha,\beta\}}}}+\frac{C}{(1+s)^{\frac{3}{4}\frac{d}{\max\{\alpha,\beta\}}}},
\end{align*}
Thus, using
$$
\int_{0}^{t/2}\frac{C}{(t-s)^{\frac{d}{2\alpha}}(1+s)^{\frac{d}{\max\{\alpha,\beta\}}}}ds\leq \frac{C}{t^{\frac{d}{2\alpha}}}\int_{0}^{t/2}\frac{C}{(1+s)^{\frac{d}{\max\{\alpha,\beta\}}}}ds\leq\frac{C}{t^{\frac{d}{2\alpha}}}\frac{C}{(1+t)^{\frac{d}{\max\{\alpha,\beta\}}-1}}
$$
\begin{align*}
\|u(t)-e^{-t\Lambda^{\alpha}}u_0\|_{L^2}&\leq \frac{C}{(1+t)^{\frac{d}{\max\{\alpha,\beta\}}-1}t^{\frac{d}{2\alpha}}}+\frac{C}{(1+t)^{\frac{d-1}{\max\{\alpha,\beta\}}-1}}+\frac{C}{(1+t)^{\frac{3}{4}\frac{d}{\max\{\alpha,\beta\}}-1}}\\
&\leq \frac{C}{(1+t)^{\frac{d-1}{\max\{\alpha,\beta\}}-1}}
\end{align*}
In the same way,
\begin{align*}
\|v(t)-e^{-t\Lambda^{\beta}}v_0\|_{L^2}&\leq \frac{C}{(1+t)^{\frac{d-1}{\max\{\alpha,\beta\}}-1}}
\end{align*}

\appendix

\section{Inequalities for the fractional Laplacian}\label{sec7}
In this appendix we recall several inequalities involving the fractional Laplacian.
\begin{lem}\label{lemaaux3}
Let $h\in \mathcal{S}(\RR^d)$ be a Schwartz function. We write $h(x^*):=\max_x h(x)$, $h(x_*):=\min_x h(x)$ and 
$$
\|h\|_{L^p\cap L^\infty}:=\max\{\|h\|_{L^p},\|h\|_{L^\infty}\}.
$$ 
Then
\begin{itemize}
\item if $h(x^*)>0$,
$$
\Lambda^\alpha h(x^*)\geq c(d,\alpha,p)\frac{h(x^*)^{1+\alpha p/d}}{\|h\|^{\alpha p/d}_{L^p}},
$$
\item if $h(x_*)<0$,
$$
\Lambda^\alpha h(x_*)\leq c(d,\alpha,p)\frac{h(x_*)|h(x_*)|^{\alpha p/d}}{\|h\|^{\alpha p/d}_{L^p}},
$$
\end{itemize}
These bounds implies the norm
$$
\|e^{-\Lambda^\alpha t}\|_{L^p\cap L^\infty (\RR^d)\rightarrow L^\infty (\RR^d)}\leq \frac{C(d,\alpha,p)}{(t +1)^{d/(\alpha p)}}.
$$
\end{lem}
\begin{proof}\textbf{Step 1;} Let's assume that $h$ takes both signs. Then we have $h(x^*)=\max_x h(x)>0$. We take $r>0$ a positive number and define 
$$
\mathcal{U}_1=\{\eta\in B(0,r) \;s.t.\; h(x^*)-h(x^*-\eta)>h(x^*)/2 \},
$$
and $\mathcal{U}_2=B(0,r)-\mathcal{U}_1$. We have
$$
\|h\|_{L^p}^p=\int_{\RR^d}|h(x^*-\eta)|^p d\eta\geq \int_{\mathcal{U}_2}|h(x^*-\eta)|^pd\eta\geq\frac{ |h(x^*)|^p}{2^p}|\mathcal{U}_2|,
$$
so,
\begin{equation}\label{eqappaux}
-\left(\frac{2\|h\|_{L^p}}{|h(x^*)|}\right)^p\leq -|\mathcal{U}_2|.
\end{equation}
\begin{eqnarray*}
\Lambda^\alpha h(x^*)&=&c_{\alpha,d}\text{P.V.}\int_{\RR^d}\frac{h(x^*)-h(x^*-\eta)}{|\eta|^{d+\alpha}}d\eta\\
&\geq& c_{\alpha,d}\text{P.V.}\int_{\mathcal{U}_1}\frac{h(x^*)-h(x^*-\eta)}{|\eta|^{d+\alpha}}d\eta\\
&\geq& c_{\alpha,d}\frac{h(x^*)}{2r^{d+\alpha}}|\mathcal{U}_1|\\
&\geq& c_{\alpha,d}\frac{h(x^*)}{2r^{d+\alpha}}\left(\omega_d r^d-|\mathcal{U}_2|\right)\\
&\geq& c_{\alpha,d}\frac{h(x^*)}{2r^{d+\alpha}}\left(\omega_d r^d-\left(\frac{2\|h\|_{L^p}}{h(x^*)}\right)^p\right),
\end{eqnarray*}
where we have used
$$
|B(0,r)|-|\mathcal{U}_2|=|\mathcal{U}_1|.
$$
We take $r$ such that
$$
\omega_d r^d=2\left(\frac{2\|h\|_{L^p}}{h(x^*)}\right)^p,
$$
thus
$$
\Lambda^\alpha h(x^*)\geq c_{\alpha,d}\frac{h(x^*)2^p\left(\frac{\|h\|_{L^p}}{h(x^*)}\right)^p}{2\left(\left(\frac{2\|h\|_{L^p}}{h(x^*)}\right)^{p/d}\left(\frac{2}{\omega_d}\right)^{1/d}\right)^{d+\alpha}}=c(d,\alpha,p)\frac{h(x^*)^{1+\alpha p/d}}{\|h\|^{\alpha p/d}_{L^p}}.
$$
\textbf{Step 2;} We have $h(x_*)=\min_x h(x)<0$. As before, we take $r>0$ a positive number and define 
$$
\mathcal{U}_1=\{\eta\in B(0,r) \;s.t.\; h(x^*)-h(x^*-\eta)<h(x^*)/2 \},
$$
and $\mathcal{U}_2=B(0,r)-\mathcal{U}_1$. In the same way, we obtain inequality \eqref{eqappaux}.
With the appropriate choice of $r$, we get
\begin{eqnarray*}
\Lambda^\alpha h(x_*)&\leq& c_{\alpha,d}\frac{h(x_*)}{2r^{d+\alpha}}\left(\omega_d r^d-\left(\frac{2\|h\|_{L^p}}{|h(x_*)|}\right)^p\right)\\
&\leq& c(d,\alpha,p)\frac{h(x_*)|h(x_*)|^{\alpha p/d}}{\|h\|^{\alpha p/d}_{L^p}}.
\end{eqnarray*}
\textbf{Step 3;} Now, we have
$$
\frac{d}{dt}\|e^{-\Lambda^\alpha t}h\|_{L^\infty}\leq -c(d,\alpha,p)\frac{\|e^{-\Lambda^\alpha t}h\|_{L^\infty}^{1+\alpha p/d}}{\|h\|^{\alpha p/d}_{L^p}},
$$
and, integrating,
$$
\|e^{-\Lambda^\alpha t}h\|_{L^\infty}\leq C(d,\alpha,p)\frac{\max\{\|h\|_{L^p},\|h\|_{L^\infty}\}}{(t +1)^{d/(\alpha p)}}.
$$
\end{proof}

\section{Commutator estimates}\label{sec8}
We prove now a commutator estimate akin to the one in \cite{Chae2singularSQG}:
\begin{lem}\label{commutator}Fix $s\geq0$. Then the following estimate holds true
$$
\|[\Lambda^s\nabla,g]f\|_{L^2}\leq C\left(\|\Lambda^s f\|_{L^2}\|\widehat{\Lambda g}\|_{L^1}+\|\Lambda^{s+1} g\|_{L^2}\|\widehat{f}\|_{L^1}\right).
$$
\end{lem}
\begin{proof}
The proof is similar to the one in \cite{Chae2singularSQG}. After taking the Fourier transform and using the inequality 
$$
|\chi|^s\leq 2^{s-1}\left(|\chi-\xi|^s+|\xi|^s\right),
$$
we have
\begin{align*}
|\widehat{[\Lambda^s\nabla,g]f]}(\chi)|&\leq C\left(\int_{\RR^d}|\chi-\xi|^{s}|\hat{f}(\chi-\xi)||\xi|| \hat{g}(\xi)|d\xi\right.\\
&\quad\left.+\int_{\RR^d}|\hat{f}(\chi-\xi)||\xi|^{1+s} |\hat{g}(\xi)|d\xi\right).
\end{align*}
Then we conclude via Plancherel's Theorem and Young's inequality for convolutions.
\end{proof}
We also recall the classical Kato-Ponce commutator estimate
\begin{lem}\label{commutator2}Fix $s>0$ and  $1<p<\infty$. Then the following estimate holds true
$$
\|[\Lambda^s,g]f\|_{L^p}\leq C\left(\|\nabla g\|_{L^{p_1}}\|\Lambda^{s-1} f\|_{L^{p_2}}+\|\Lambda^{s} g\|_{L^{p_3}}\|f\|_{L^{p_4}}\right),
$$
for
$$
\frac{1}{p}=\frac{1}{p_1}+\frac{1}{p_2}=\frac{1}{p_3}+\frac{1}{p_4},
1<p_1,p_4\leq \infty,\;1<p_2,p_3<\infty.
$$
\end{lem}
\section{Interpolation inequalities for the Wiener's algebra}\label{sec9}
In this appendix we recall and prove several inequalities involving fractional Sobolev and the Wiener's algebra that may be interesting by themselves. 
\begin{lem}\label{lemmainterpolation}Assume that
$$
u\in L^1(\RR^d)\cap L^\infty(\RR^d)\cap \dot{H}^{d/2+\delta}(\RR^d).
$$
Then the following inequalities hold
\begin{align}\label{eq:10}
\|\hat{u}\|_{L^1(\RR^d)}&\leq C\|u\|_{L^1(\RR^d)}^{\frac{\delta}{d+\delta}}\|u\|_{\dot{H}^{d/2+\delta}(\RR^d)}^{\frac{d}{d+\delta}},\,\forall\, \delta>0,\\
\label{eq:10b}
\|\hat{u}\|_{L^1(\RR^d)}&\leq C\|u\|_{L^2(\RR^d)}^{\frac{\delta}{d/2+\delta}}\|u\|_{\dot{H}^{d/2+\delta}(\RR^d)}^{\frac{d/2}{d/2+\delta}},\,\forall\, \delta>0.
\end{align}
\end{lem}
\begin{proof}
We have
\begin{align*}
\|\hat{u}\|_{L^1(\RR^d)}&=\int_{|\xi|<R}|\hat{u}(\xi)|d\xi+\int_{|\xi|>R}\frac{|\xi|^{d/2+\delta}}{|\xi|^{d/2+\delta}}|\hat{u}(\xi)|d\xi\nonumber\\
&\leq \|u\|_{L^1(\RR^d)}R^d\frac{\pi^{d/2}}{\Gamma\left(\frac{d}{2}+1\right)}+\|u\|_{\dot{H}^{d/2+\delta}(\RR^d)}\sqrt{\int_{|\xi|>R}|\xi|^{-d-2\delta}d\xi}\\
&\leq \|u\|_{L^1(\RR^d)}R^d\frac{\pi^{d/2}}{\Gamma\left(\frac{d}{2}+1\right)}+\|u\|_{\dot{H}^{d/2+\delta}(\RR^d)}\sqrt{\int_{R}^\infty r^{d-1-d-2\delta}dr}\\
&\leq \|u\|_{L^1(\RR^d)}R^d\frac{\pi^{d/2}}{\Gamma\left(\frac{d}{2}+1\right)}+\|u\|_{\dot{H}^{d/2+\delta}(\RR^d)}C_\delta R^{-\delta}
\end{align*}
With the choice 
$$
R=\left(\frac{\|u\|_{\dot{H}^{d/2+\delta}}}{\|u\|_{L^1}}\right)^{\frac{1}{d+\delta}}
$$ 
and we conclude the inequality \eqref{eq:10}. To prove the second inequality \eqref{eq:10b}, we compute
\begin{align*}
\|\hat{u}\|_{L^1(\RR^d)}&=\int_{|\xi|<R}|\hat{u}(\xi)|d\xi+\int_{|\xi|>R}\frac{|\xi|^{d/2+\delta}}{|\xi|^{d/2+\delta}}|\hat{u}(\xi)|d\xi\\
&\leq \|\hat{u}\|_{L^2(\RR^d)}\sqrt{R^d\frac{\pi^{d/2}}{\Gamma\left(\frac{d}{2}+1\right)}}+\|u\|_{\dot{H}^{d/2+\delta}(\RR^d)}C_\delta R^{-\delta}.
\end{align*}
Now, we can take 
$$
R=\left(\frac{\|u\|_{\dot{H}^{d/2+\delta}}}{\|u\|_{L^1}}\right)^{\frac{1}{d/2+\delta}}
$$
\end{proof}


\subsection*{Acknowledgment}
The author is partially supported by the Grant MTM2014-59488-P from the former Ministerio de Econom\'ia y Competitividad (MINECO, Spain).

\end{document}